\documentclass[11pt]{article}
\usepackage{amssymb}
\usepackage{amsmath}
\usepackage{amsfonts}
\usepackage{graphicx}

\newtheorem{theorem}{Theorem}[section]

\newtheorem{corollary}[theorem]{Corollary}

\newtheorem{definition}[theorem]{Definition}
\newtheorem{example}[theorem]{Example}

\newtheorem{lemma}[theorem]{Lemma}

\newtheorem{proposition}[theorem]{Proposition}
\newtheorem{remark}[theorem]{Remark}

\newenvironment{proof}[1][Proof]{\noindent\textbf{#1.} }{\ \rule{0.5em}{0.5em}}
\input{tcilatex}
\begin{document}

\author{Yulia Kempner \\
%EndAName
Department of Computer Science\\
Holon Institute of Technology, Israel\\
yuliak@hit.ac.il \and Vadim E. Levit \\
%EndAName
Department of Computer Science \\
Ariel University, Israel\\
levitv@ariel.ac.il}
\date{}
\title{Cospanning characterizations of violator and co-violator spaces }
\maketitle

\begin{abstract}
Given a finite set $E$ and an operator $\sigma :2^{E}\longrightarrow 2^{E}$,
two subsets $X,Y\subseteq E$ are \textit{cospanning} if $\sigma (X)=\sigma
(Y)$ (Korte, Lovasz, Schrader; 1991). We investigate cospanning relations on
violator spaces. A notion of a violator space was introduced in (G\"{a}%
rtner, Matou\v{s}ek, R\"{u}st, \v{S}kovro\v{n}by; 2008) as a combinatorial
framework that encompasses linear programming and other geometric
optimization problems.

Violator spaces are defined by violator operators. We introduce \textit{%
co-violator spaces} based on contracting operators known also as choice
functions. Let $\alpha ,\beta :2^{E}\longrightarrow 2^{E}$ be a violator
operator and a co-violator operator, respectively. Cospanning
characterizations of violator spaces allow us to obtain some new properties
of violator operators, co-violator operators, and their interconnections. In
particular, we show that uniquely generated violator spaces enjoy so-called
Krein-Milman properties, i.e., $\alpha (\beta \left( X\right) )=\alpha (X)$
and $\beta \left( \alpha \left( X\right) \right) =\beta \left( X\right) $
for every $X\subseteq E$.

\end{abstract}

\section{Introduction}

Each set operator determines the partition of sets to equivalence classes
with equal value of the operator. Let us have some set operator $\alpha $.
Following \cite{Greedoids} we call two sets $X,Y$ \textit{cospanning} if $%
\alpha (X)=\alpha (Y)$. Thus each set operator generates the cospanning
equivalence relation on sets. Our goal is to investigate cospanning
relations on violator spaces. These spaces were introduced in order to
develop a combinatorial framework encompassing linear programming and other
geometric optimization problems \cite{VS}. Violator spaces are defined by
violator operators, which generalize closure operators \cite{KempnerLevit}.
We also pay special attention to violator spaces with unique bases. In
Section 2, we introduce co-violator spaces based on contracting operators
known also as choice functions. In Section 3, we characterize the cospanning
relation with regards to violator spaces and describe the equivalence
classes of the relation for violator and co-violator spaces. Cospanning
characterizations allow us to obtain some new properties of violator
operators, co-violator operators and their interconnections. In particular,
we show that uniquely generated violator spaces enjoy so-called Krein-Milman
properties.

\subsection{ Violator spaces}

Violator spaces are arisen as a generalization of Linear Programming
problems. LP-type problems have been introduced and analyzed by Matou\v{s}%
ek, Sharir and Welzl \cite{MSW, SW} as a combinatorial framework that
encompasses linear programming and other geometric optimization problems.
Further, Matou\v{s}ek et al. \cite{VS} define a simpler framework: violator
spaces, which constitute a proper generalization of LP-type problems.
Originally, violator spaces were defined for a set of constraints $E$, where
each subset of constraints $G\subseteq E$ was associated with $\nu (G)$ -
the set of all constraints violating $G$.

The classic example of an LP-type problem is the problem of computing the
smallest enclosing ball of a finite set of points in $%
%TCIMACRO{\U{211d} }%
%BeginExpansion
\mathbb{R}
%EndExpansion
^{d}$. Here $E$ is a set of points in $%
%TCIMACRO{\U{211d} }%
%BeginExpansion
\mathbb{R}
%EndExpansion
^{d}$, and the violated constraints of some subset of the points $G$ are
exactly the points lying outside the smallest enclosing ball of $G$.

\begin{definition}
\cite{VS} A \textit{violator space} is a pair $(E,\nu)$, where $E$ is a
finite set and $\nu$ is a mapping $2^{E}\rightarrow2^{E}$ such that for all
subsets $X,Y\subseteq E$ the following properties are satisfied:

\textbf{V11}: $X \cap\nu(X) = \emptyset$ (consistency),

\textbf{V22}: $(X\subseteq Y$ and $Y \cap\nu(X) = \emptyset) \Rightarrow
\nu(X)=\nu(Y)$ (locality).
\end{definition}

Let $(E,\nu)$ be a violator space. Define $\varphi(X)=E-\nu(X)$. In what
follows, if $(E,\nu)$ is a violator space and $\varphi(X)=E-\nu(X)$, then $%
(E, \varphi)$ will be called a violator space as well.

\begin{definition}
(\cite{KempnerLevit}) A \textit{violator space} is a pair $(E,\varphi)$,
where $E$ is a finite set and $\varphi$ is an operator $2^{E}%
\rightarrow2^{E} $ such that for all subsets $X,Y\subseteq E$ the following
properties are satisfied:

\textbf{V1}: $X\subseteq\varphi(X)$ (extensivity),

\textbf{V2}: $(X\subseteq Y\subseteq\varphi(X))\Rightarrow\varphi
(X)=\varphi(Y)$ (self-convexity).
\end{definition}

Each violator operator $\varphi$ is idempotent. Indeed, extensivity implies $%
X\subseteq\varphi(X)\subseteq\varphi(X)$. Then, by self-convexity, we
conclude with $\varphi(\varphi(X))=\varphi(X)$.

\begin{lemma}
(\cite{KempnerLevit}) \label{un} Let $(E,\varphi)$ be a violator space. Then 
\begin{equation}
\varphi(X)=\varphi(Y)\Rightarrow\varphi(X\cup Y)=\varphi(X)=\varphi(Y)
\label{Union}
\end{equation}
and 
\begin{equation}
(X \subseteq Y \subseteq Z) \wedge (\varphi(X)=\varphi(Z)) \Rightarrow
\varphi(X)=\varphi(Y)=\varphi(Z)  \label{Convexity}
\end{equation}
for every $X,Y,Z\subseteq E$.
\end{lemma}

Since the second property deals with all sets lying between two given sets,
following \cite{Monjardet} we call the property \textit{convexity}.

\subsection{Uniquely generated violator spaces}

Let $(E,\alpha)$ be an arbitrary space with the operator $%
\alpha:2^{E}\rightarrow2^{E}$. $B\subseteq E$ is a \textit{generator} of $%
X\subseteq E$ if $\alpha(B)=\alpha(X)$. For $X\subseteq E$, a \textit{basis}
(minimal generator) of $X$ is a inclusion-minimal set $B\subseteq E$ (not
necessarily included in $X$) with $\alpha(B)=\alpha(X)$. A space $(E,\alpha)$
is \textit{uniquely generated} if every set $X\subseteq E$ has a unique
basis.

\begin{proposition}
\cite{KempnerLevit} \label{UQP} A violator space $(E,\varphi)$ is uniquely
generated if and only if for every $X,Y\subseteq E$ 
\begin{equation}
\varphi(X)=\varphi(Y)\Rightarrow\varphi(X\cap Y)=\varphi(X)=\varphi(Y)
\label{UQ}
\end{equation}
\end{proposition}

We can rewrite the property (\ref{UQ}) as follows: for every set $X\subseteq
E$ of a uniquely generated violator space $(E,\varphi)$, the basis $B$ of $X$
is the intersection of all generators of $X$: 
\begin{equation}
B=\bigcap\{Y\subseteq E:\varphi(Y)=\varphi(X)\}.  \label{UQI}
\end{equation}

One of the known examples of a not uniquely generated violator space is the
violator space associated with the smallest enclosing ball problem. A basis
of a set of points is a minimal subset with the same enclosing ball. In
particular, all points of the basis are located on the ball's boundary. For $%
%TCIMACRO{\U{211d} }%
%BeginExpansion
\mathbb{R}
%EndExpansion
^{2}$ the set $X$ of the four corners of a square has two bases: the two
pairs of diagonally opposite points. Moreover, one of these pairs is a basis
of the second pair. Thus the equality (\ref{UQI}) does not hold.

For each arbitrary space $(E,\alpha)$ with the operator $\alpha:2^{E}%
\rightarrow2^{E}$, an element $x$ of a subset $X\subseteq E$ is \textit{an
extreme point} of $X$ if $x\notin\alpha(X-x)$. The set of extreme points of $%
X$ is denoted by $ex(X)$.

\begin{proposition}
\cite{KempnerLevit} \label{exp} Let $(E,\varphi)$ be a violator space. Then 
\begin{equation*}
ex(X)=\bigcap\{B\subseteq X:\varphi(B)=\varphi(X)\}.
\end{equation*}
\end{proposition}

\begin{proposition}
\cite{KempnerLevit} \label{exp-vs} Let $(E,\varphi)$ be a violator space.
Then 
\begin{equation*}
ex(\varphi(X))\subseteq ex(X).
\end{equation*}
\end{proposition}

\begin{theorem}
\cite{KempnerLevit} \label{th2} Let $(E,\varphi)$ be a violator space. Then $%
(E,\varphi)$ is uniquely generated if and only if for every set $X\subseteq
E $, $\varphi (X)=\varphi(ex(X))$.
\end{theorem}

\begin{corollary}
\cite{KempnerLevit} \label{cor_ex} Let $(E,\varphi)$ be a uniquely generated
violator space. Then for every $X\subseteq E$ the set $ex(X)$ is the unique
basis of $X$.
\end{corollary}

\section{Co-violator spaces}

Theorem \ref{th2} and Proposition \ref{exp-vs} show that there is some
duality between extensive ($X \subseteq \varphi(X)$) and contracting ($ex(X)
\subseteq X$) operators. To study this connection we introduce a new type of
spaces.

\begin{definition}
A \textit{co-violator space} is a pair $(E,c)$, where $E$ is a finite set
and $c$ is an operator $2^{E}\rightarrow2^{E}$ such that for all subsets $%
X,Y\subseteq E$ the following properties are satisfied:

\textbf{CV1}: $c(X)\subseteq X$,

\textbf{CV2}: $(c(X)\subseteq Y\subseteq X)\Rightarrow c(X)=c(Y)$.
\end{definition}

Operators satisfying the property \textbf{CV1} are called contracting
operators.

In social sciences, contracting operators are called choice functions,
usually adding a requirement that $c(X)\neq \emptyset $ for every $X\neq
\emptyset $. The property \textbf{CV2} is called the \textit{outcast property%
} or the \textit{Aizerman property} \cite{Monjardet}.

The properties of co-violator spaces correspond to the corresponding
("mirrored") properties of violator spaces. For instance, every co-violator
operator $c$ is idempotent. Indeed, since $c$ is contracting $c(X)\subseteq
c(X)\subseteq X$. Then, \textbf{CV2} implies $c(c(X))=c(X)$.

Lemma \ref{un} is converted to the following.

\begin{lemma}
\label{co-un} Let $(E,c)$ be a co-violator space. Then 
\begin{equation}
c(X)=c(Y)\Rightarrow c(X\cap Y)=c(X)=c(Y)  \label{Intersection}
\end{equation}
and 
\begin{equation}
(X \subseteq Y \subseteq Z) \wedge (c(X)=c(Z)) \Rightarrow c(X)=c(Y)=c(Z)
\label{Co-Convexity}
\end{equation}
for every $X,Y,Z\subseteq E$.
\end{lemma}

\begin{proof}
Prove (\ref{Intersection}). Let $c(X)=c(Y)$. \textbf{CV1} implies that $%
c(X)\subseteq X$ and $c(Y)=c(X) \subseteq Y$. Then $c(X) \subseteq X\cap Y
\subseteq X$, that gives (by \textbf{CV2}) $c(X\cap Y)=c(X)$.

To prove (\ref{Co-Convexity}) let $(X\subseteq Y\subseteq
Z)\wedge(c(X)=c(Z)) $. \textbf{CV1} yields $c(Z)=c(X)\subseteq X\subseteq Y$%
. Then outcast property allows us to get $c(Z)\subseteq Y\subseteq
Z\Rightarrow c(Y)=c(X)=c(Z)$.
\end{proof}

It is easy to see that all the properties of violator spaces hold in their
dual interpretation for co-violator spaces. Since a co-violator operator is
a choice function with outcast properties, the connection between these two
types of spaces may result in better understanding of two theories and in
new findings in each of them.

Connections between contracting and extensive operators were studied in many
works, while most of them were dedicated to connections between choice
functions and closure operators \cite{Ando, Danilov, Monjardet}. Naturally,
extreme point operators were considered as choice functions. But, as we will
see in Proposition \ref{outcast_u}, the extreme point operator of a violator
space satisfies the outcast property, and so it forms a co-violator space,
if and only if the violator space is uniquely generated. We also consider
choice functions investigated in \cite{Libkin}. The \textit{interior
operator }(well-known in topology) is dual to a closure operator. Given an
extensive operator $\varphi :2^{E}\rightarrow 2^{E}$, one can get a
contracting operator $c$: $c(X)=E-\varphi (E-X)$ or $\overline{c(X)}=\varphi
(\overline{X})$.

\begin{proposition}
\label{co-operator} $(E,\varphi)$ is a violator space if and only if $(E,c)$
is a co-violator space, where $c(X)=\overline{\varphi(\overline{X})}$.
\end{proposition}

\begin{proof}
It is easy to see that $\varphi $ is an extensive operator if and only if $c$
is a contracting operator. To prove that $c$ satisfies the outcast property
if and only if $\varphi $ is self-convex one has just to pay attention that:

$c(X) \subseteq Y \subseteq X \Leftrightarrow \overline{X} \subseteq 
\overline{Y} \subseteq \overline{c(X)} \Leftrightarrow \overline{X}
\subseteq \overline{Y} \subseteq \varphi(\overline{X}) \Rightarrow \varphi(%
\overline{X})=\varphi(\overline{Y}) \Leftrightarrow c(X)=c(Y)$. The opposite
direction is proved completely analogously.
\end{proof}

\section{Cospanning relations of violator and co-violator spaces}

Let $E=\left\{ {x_{1},x_{2},...,x_{d}}\right\} $. The graph $H(E)$ is
defined as follows. The vertices are the finite subsets of $E$, two vertices 
$A$ and $B$ are adjacent if and only if they differ in exactly one element.
Actually, $H(E)$ is \textit{the hypercube} on $E$ of dimension $d$, since
the hypercube is known to be equivalently considered as the graph on the
Boolean space $\{0,1\}^{d}$ in which two vertices form an edge if and only
if they differ in exactly one position.

Let $(E,\varphi)$ be a violator space. The two sets $X$ and $Y$ are \textit{%
equivalent} (or \textit{cospanning}) if $\varphi(X)=\varphi(Y)$. In what
follows, $\mathcal{P}$ denotes a partition of $H(E)$ ( or $2^{E}$) into
equivalence classes with regard to this relation, and $[A]_{\varphi}:=\{X%
\subseteq E:\varphi(X)=\varphi(A)\}$.

\begin{remark}
Note, that the cospanning relation associated with a violator operator $%
\varphi$ coincides with the cospanning relation associated with an original
violator mapping $\nu$.
\end{remark}

The following theorem characterizes cospanning relations in violator spaces.

\begin{theorem}
\label{T_rel} Let $E$ be a finite set and $R \subseteq 2^{E} \times 2^{E}$
be an equivalence relation on $2^{E}$. Then $R$ is the cospanning relation
of a violator space if and only if the following properties hold for every $%
X,Y,Z \subseteq E$:

\textbf{R1}: if $(X,Y) \in R$, then $(X,X \cup Y) \in R$

\textbf{R2}: if $X \subseteq Y \subseteq Z$ and $(X,Z) \in R$, then $(X,Y)
\in R$.
\end{theorem}

\begin{proof}
Necessity follows immediately from Lemma \ref{un}.

Let us define an operator $\varphi $ and prove that it satisfies extensivity
and self-convexity. Since $R$ is an equivalence relation, it defines a
partition of $2^{E}$. Then, for each $X\subseteq E$ there is only one class
containing $X$. Thus for every set $X$, we define $\varphi (X)$ as a maximal
element in the class $[X]_{R}$. Notice, that the property \textbf{R1}
implies that each equivalence class has a unique maximal element, so the
partition is well-defined. Hence, we obtain that $X\subseteq \varphi (X)$
and $\varphi (\varphi (X))=\varphi (X)$. Then the self-convexity follows
immediately from \textbf{R2}. It is easy to see that the cospanning relation
w.r.t. $\varphi $ coincides with $R$.
\end{proof}

In conclusion, each equivalence class of the cospanning relation of a
violator space is closed under union (\textbf{R1}) and convex (\textbf{R2}).

The following theorem characterizes equivalence classes of co-violator
spaces.

\begin{theorem}
\label{cv_rel} Let $E$ be a finite set and $R \subseteq 2^{E} \times 2^{E}$
be an equivalence relation on $2^{E}$. Then $R$ is the cospanning relation
of a co-violator space if and only if the following properties hold for
every $X,Y,Z \subseteq E$:

\textbf{R3}: if $(X,Y) \in R$, then $(X,X \cap Y) \in R$

\textbf{R2}: if $X \subseteq Y \subseteq Z$ and $(X,Z) \in R$, then $(X,Y)
\in R$.
\end{theorem}

\begin{proof}
Necessity follows immediately from Lemma \ref{co-un}. By analogy with the
proof of Theorem \ref{T_rel} we define $c(X)$ to be a minimal element in the
class $[X]_{R}$. Since each class is closed under intersection (\textbf{R3}%
), the partition is well-defined. It is easy to see that operator $c$ is
contracting, satisfies the outcast property, and its cospanning relation
coincides with $R$.
\end{proof}

Consider now both a violator operator $\varphi $ and a co-violator operator $%
c(X)=\overline{\varphi (\overline{X})}$.

\begin{proposition}
\label{co-co} There is a one-to-one correspondence between an equivalence
class $[X]_{\varphi}$ of $X$ of the cospanning relation associated with a
violator operator $\varphi$ and an equivalence class $[\overline{X}]_{c}$
w.r.t. a co-violator operator $c$, i.e., $A \in [X]_{\varphi}$ if and only
if $\overline{A} \in [\overline{X}]_{c}$.
\end{proposition}

\begin{proof}
Indeed, $A \in [X]_{\varphi} \Leftrightarrow \varphi(X)=\varphi(A)
\Leftrightarrow \overline{c(\overline{X})} = \overline{c(\overline{A})}
\Leftrightarrow c(\overline{X})=c(\overline{A}) \Leftrightarrow \overline{A}
\in [\overline{X}]_{c}$.
\end{proof}

A uniquely generated violator space defines a cospanning relation with
additional property \textbf{R3} (see Proposition \ref{UQP}).

All in all, every uniquely generated violator space is a co-violator space
as well. Each equivalence class of the cospanning relation of a uniquely
generated violator space has an unique minimal element and an unique maximal
element. More precisely, for the sets $A\subseteq B\subseteq E$, let us
define the \textit{interval} $[A,B]$ as $\{C\subseteq E:A\subseteq
C\subseteq B\}$. Then each equivalence class of an uniquely generated
violator space is an interval. We call a partition of $H(E)$ into disjoint
intervals a \textit{hypercube partition}. The following Theorem follows
immediately from Theorem \ref{T_rel} and Proposition \ref{UQP}.

\begin{theorem}
(\cite{Clarkson}) (i) If $(E,\varphi)$ is a uniquely generated violator
space, then $\mathcal{P}$ is a hypercube partition of $H(E)$.

(ii) Every hypercube partition is the partition $\mathcal{P}$ of $H(E)$ into
equivalence classes of a uniquely generated violator space.
\end{theorem}

More specifically \cite{KempnerLevit}, $[A]_{\varphi}=[ex(A),\varphi(A)]$
for every set $A \subseteq E$.

Let us consider now a uniquely generated violator space $(E,\varphi )$ and
the operator $ex$. Since each equivalence class $[A]_{\varphi }$ w.r.t.
operator $\varphi $ is an interval $[ex(A),\varphi (A)]$, we can see that
for each $X\in \lbrack ex(A),\varphi (A)]$ not only $\varphi (X)=\varphi (A)$%
, but $ex(X)=ex(A)$ as well. Since $\mathcal{P}$ is a hypercube partition of 
$H(E)$ we conclude with $[X]_{\varphi }=[X]_{ex}$. Thus the cospanning
partition (quotient set) associated with an operator $\varphi $ coincides
with the cospanning partition associated with a contracting operator $ex$.
Since $ex(X)$ is a minimal element of $[X]$ we immediately obtain the
following

\begin{proposition}
If $(E,\varphi)$ is a uniquely generated violator space, then operator $ex$
satisfies the following properties:

\textbf{X1}: $ex(ex(X))=ex(X)$

\textbf{X2}: $ex(X)=ex(Y)\Rightarrow ex(X\cup Y)=ex(X)=ex(Y)$

\textbf{X3}:$(X \subseteq Y \subseteq Z) \wedge (ex(X)=ex(Z)) \Rightarrow
ex(X)=ex(Y)=ex(Z)$

\textbf{X4}: $ex(X)=ex(Y)\Rightarrow ex(X\cap Y)=ex(X)=ex(Y)$
\end{proposition}

If $(E,\varphi )$ is not a uniquely generated violator space, then the
operator $ex$ may or may not satisfy the properties \textbf{X1}-\textbf{X4}.
Consider the two following examples.

\begin{example}
Let $E=\{1,2,3\}$. Define $\varphi(X)=X$ for each $X\subseteq E$ except $%
\varphi(\{2\})=\varphi(\{3\})=\{2,3\}$ and $\varphi(\{1,2\})=\varphi(\{1,3%
\})=\{1,2,3\}$. It is easy to check that $(E,\varphi)$ is a violator space
and the operator $ex$ satisfies \textbf{X1},\textbf{X2}, and \textbf{X4},
but while $ex(\{1\})=ex(\{1,2,3\})=\{1\}$, $ex(\{1\}) \neq ex(\{1,2\})$,
i.e., the operator $ex$ is not convex.
\end{example}

\begin{example}
Let $E=\{1,2,3,4,5,6\}$. Define $\varphi (X)=X$ for each $X\subseteq E$
except $\varphi (\{1\})=\{1,2\}$, $\varphi (\{1,2,3\})=\varphi
(\{1,2,4\})=\{1,2,3,4\}$ and $\varphi (\{1,2,5\})=\varphi
(\{1,2,6\})=\{1,2,5,6\}$. It is easy to check that $(E,\varphi )$ is a
violator space. In addition, $ex(\{1,2,3,4\})=ex(\{1,2,5,6\})=\{1,2\}$,
while $ex(\{1,2\})=\{1\}$. Hence, $ex$ is not idempotent (\textbf{X1}) and
does not satisfy \textbf{X4}. Since $ex(\{1,2,3,4,5,6\})=\{1,2,3,4,5,6\}$
the operator $ex$ does not satisfy \textbf{X2} as well, but, compared to the
previous example, $ex$ is convex.
\end{example}

\begin{proposition}
\label{outcast_u} Let $(E,\varphi)$ be a violator space. The following
assertions are equivalent:

(i) $(E,\varphi)$ is uniquely generated

(ii) \textbf{X5}: $(ex(X) \subseteq Y \subseteq X) \Rightarrow ex(X)=ex(Y)$
(the outcast property)

(iii) \textbf{X6}: $\varphi(ex(X))=\varphi(X)$

(iv) \textbf{X7}: $ex(\varphi(X))=ex(X)$
\end{proposition}

\begin{proof}
If $(E,\varphi)$ is a uniquely generated violator space, then operator $ex$
satisfies \textbf{X5},\textbf{X6} and \textbf{X7}, since $%
[X]_{\varphi}=[X]_{ex}=[ex(X),\varphi(X)]$.

Before we continue with the proof, it is important to mention that from the
definition of the operator $ex$ it follows that $ex(B)=B$ for each basis $B$.

Further we prove that if a violator space $(E,\varphi )$ satisfies the
property \textbf{X5}, then it is uniquely generated. Suppose that there is a
set $X\subseteq E$ with two bases $B_{1}$ and $B_{2}$. Then $\varphi
(X)=\varphi (B_{1})=\varphi (B_{2})=\varphi (B_{1}\cup B_{2})$. Thus
Proposition \ref{exp} implies $ex(B_{1}\cup B_{2})\subseteq B_{1}\cap B_{2}$%
. Then we have $ex(B_{1}\cup B_{2})\subseteq B_{1}\subseteq B_{1}\cup B_{2}$
and $ex(B_{1}\cup B_{2})\subseteq B_{2}\subseteq B_{1}\cup B_{2}$, but $%
ex(B_{1})=B_{1}\neq ex(B_{2})=B_{2}$. In other words, we see that $ex$ does
not satisfy the outcast property.

$(iii)\Rightarrow (i)$ follows from Theorem \ref{th2}.

Now, it is only left to prove that if a violator space $(E,\varphi )$
satisfies the property \textbf{X7}, then it is uniquely generated. Suppose
there is a set $X\subseteq E$ with two bases $B_{1}\neq B_{2}$. Then $%
\varphi (X)=\varphi (B_{1})=\varphi (B_{2})$, and so $ex(\varphi
(B_{1}))=ex(\varphi (B_{2}))$. Since $ex(B_{1})=B_{1}\neq ex(B_{2})=B_{2}$,
we conclude that the property \textbf{X7} does not hold.
\end{proof}

It is worth reminding that \textbf{X6} and \textbf{X7} are called
Krein-Milman properties. In other words, every uniquely generated violator
space is a Krein-Milman space \cite{KM}.

\section{Conclusion}

Many combinatorial structures are described using operators defined on their
ground sets. For instance, closure spaces are defined by closure operators,
and violator spaces are described by violator operators. In this paper, we
introduced co-violator spaces based on contracting operators known also as
choice functions. Cospanning characterizations of violator spaces allowed us
to obtain some new properties of violator operators, co-violator operators
and their interconnections. In further research, our intent is to extend
this "cospanning" approach to a wider spectrum of combinatorial structures
closure spaces, convex geometries, antimatroids, etc.


\begin{thebibliography}{99}
\bibitem{Ando} Ando, K. \emph{Extreme point axioms for closure spaces},
Discrete Mathematics \textbf{306}(24), (2006) 3181--3188

\bibitem{Clarkson} Brise, Y., and G\"{a}rtner, B. \emph{Clarkson's algorithm
for violator spaces}, Computational Geometry \textbf{44 }(2011) 70--81.

\bibitem{Danilov} Danilov, V., and Koshevoy, G. \emph{Choice functions and
extensive operators}, Order \textbf{26}(2009) 69---94

\bibitem{Libkin} Demetrovics, J., Hencsey, G., Libkin, L, and Muchnik, I. 
\emph{On the interaction between closure operations and choice functions
with applications to relational databases}, Acta Cybernetica, \textbf{10}
(3),(1992) 129--139.

\bibitem{VS} G\"{a}rtner, B., Matou\v{s}ek, J., R\"{u}st, L. and \v{S}kovro%
\v{n}, P. \emph{Violator spaces: structure and algorithms}, Disc. Appl.
Math. \textbf{156} (2008), 2124--2141.

\bibitem{KM} Kempner, Y., and Levit, V.E. \emph{Krein-Milman spaces},
Electronic Notes in Discrete Mathematics \textbf{80} (2019), 203--213.

\bibitem{KempnerLevit} Kempner, Y., and Levit, V.E. \emph{Violator spaces vs
closure spaces}, European Journal of Combinatorics \textbf{68} (2018),
281--286.

\bibitem{Greedoids} Korte, B., Lov\'{a}sz, L. and Schrader, R.
``Greedoids,'' Springer-Verlag, New York/Berlin (1991).

\bibitem{MSW} Matou\v{s}ek, J., Sharir, M. and Welzl, E. \emph{A
subexponential bound for linear programming}. Algorithmica \textbf{16}
(1996) 498--516.

\bibitem{Monjardet} Monjardet, B. and Raderanirina, V. \emph{The duality
between the anti-exchange closure operators and the path independent choice
operators on a finite set}, Mathematical Social Sciences \textbf{41} (2001)
131--150.

\bibitem{SW} Sharir, M. and Welzl, E. \emph{A combinatorial bound for linear
programming and related problems}, Lecture Notes in Computer Science \textbf{%
577} (1992) 569--579.
\end{thebibliography}
\end{document}